\documentclass{amsart}
\usepackage{amsmath,xypic,graphics,hyperref,subfigure}

\newcommand{\excise}[1]{{}}
\newcommand{\pad}{\rule[-3mm]{0mm}{8mm}}

\theoremstyle{plain}
\newtheorem{theorem}{Theorem}
\numberwithin{theorem}{section}

\newtheorem{corollary}[theorem]{Corollary}

\theoremstyle{definition}
\newtheorem{definition}[theorem]{Definition}

\newtheorem{example}[theorem]{Example}

\theoremstyle{remark}
\newtheorem{remark}[theorem]{Remark}

\newtheorem{claim}{Claim}

\newcommand{\cc}{\mathbf{c}}
\newcommand{\onev}{\mathbf{1}}
\newcommand{\zerov}{\mathbf{0}}

\newcommand{\bd}{\partial^{\phantom{*}}} 
\newcommand{\bdo}{\partial} 

\newcommand{\rbd}{\tilde\partial^{\phantom{*}}} 
\newcommand{\rcbd}{\tilde\partial^{*}} 

\newcommand{\cbd}{\partial^*}

\newcommand{\Ltot}{L^{\rm tot}} 
\newcommand{\Lud}{L}
\newcommand{\Ldu}{L^{\rm du}}

\DeclareMathOperator{\coker}{coker}

\DeclareMathOperator{\im}{im}
\DeclareMathOperator{\rank}{rank}

\newcommand{\HH}{\tilde H} 

\newcommand{\Nn}{\mathbb{N}}
\newcommand{\Qq}{\mathbb{Q}}

\newcommand{\Zz}{\mathbb{Z}}

\newcommand{\st}{\colon}

\newcommand{\x}{\times}
\newcommand{\isom}{\cong}
\newcommand{\sm}{\backslash}

\newcommand{\SST}{{\mathcal T}}
\newcommand{\rL}{\tilde{L}}
\author{Art M.\ Duval}
\address{Department of Mathematical Sciences \\ University of Texas at El Paso}
\author{Caroline J.\ Klivans}
\address{Departments of Mathematics and Computer Science \\ The University of Chicago}
\author{Jeremy L.\ Martin}
\address{Department of Mathematics \\ University of Kansas}
\title{Critical Groups of Simplicial Complexes}
\keywords{graph, simplicial complex, critical group, combinatorial Laplacian, chip-firing game, sandpile model, spanning trees}
\subjclass[2010]{Primary 05E45; Secondary 05C50, 05C21, 55U15, 57M15}
\date{February 25, 2011}
\begin{document}

\begin{abstract}
We generalize the theory of critical groups from graphs to simplicial
complexes.  Specifically, given a simplicial complex,
we define a family of 
abelian groups 
in terms of combinatorial Laplacian operators, generalizing
the construction of the critical group of a graph.
We show how to realize these critical groups explicitly as
cokernels of reduced Laplacians, and prove that
they are finite, with orders given by
weighted enumerators of simplicial spanning trees.
We describe how the critical groups of a complex represent flow along
its faces, and sketch another potential interpretation
as analogues of Chow groups.
\end{abstract}

\maketitle
\section{Introduction}

Let $G$ be a finite, simple, undirected, connected graph.  The
\emph{critical group} of $G$ is a finite abelian group $K(G)$ whose
cardinality is the number of spanning trees of $G$.  The critical
group is an interesting graph invariant in its own right, and it
also arises naturally in the theory of a discrete dynamical system with
many essentially equivalent formulations --- the \emph{chip-firing game},
\emph{dollar game}, \emph{abelian sandpile model}, etc.--- that
has been discovered independently in contexts
including statistical physics, arithmetic geometry, and combinatorics.
There is an extensive literature on these models and their behavior:
see, e.g., \cite{Biggs,BLS,Dhar,GodRoy,LP}.  In all guises, the model
describes a certain type of discrete flow along the edges of $G$. The
elements of the critical group correspond to
states in the flow model that are stable, but for which a small
perturbation causes an instability.

The purpose of this paper is to extend the theory of the critical
group from graphs to simplicial complexes.  For a finite simplicial complex
$\Delta$ of dimension~$d$, we define its higher critical groups as
$$K_i(\Delta) := \ker \bd_i / \im(\bd_{i+1}\cbd_{i+1})$$
for $0\leq i\leq d-1$; here $\bd_j$ means the simplicial boundary map mapping $j$-chains
to $(j-1)$-chains.  The map $\bd_{i+1}\cbd_{i+1}$ is called an \emph{(updown) combinatorial
Laplacian operator}.  
For $i=0$, our definition coincides with the standard definition
of the critical group of the 1-skeleton of $\Delta$.
Our main result
(Theorem~\ref{main-thm}) states that, under certain mild assumptions
on the complex $\Delta$, the group $K_i(\Delta)$ is in fact isomorphic to the cokernel
of a reduced version of the Laplacian.  It follows from a simplicial analogue of the matrix-tree theorem
\cite{DKM,DKM2} that the orders $|K_i(\Delta)|$ of the higher critical
groups are given by a torsion-weighted enumeration of
higher-dimensional spanning trees (Corollary~\ref{count-corollary})
and in terms of the eigenvalues of the Laplacian operators
(Corollary~\ref{alt-product}).  In the case of a
simplicial sphere, we prove (Theorem~\ref{spheres})
that the top-dimensional critical
group is cyclic, with order equal to the number of
facets, generalizing the corresponding statement
\cite{Lor1,Lor2,Merris} for cycle graphs.
In the case that $\Delta$ is a skeleton of an $n$-vertex simplex,
the critical groups are direct sums of copies of $\Zz/n\Zz$;
as we discuss in Remark~\ref{Molly},
this follows from an observation of Maxwell \cite{Maxwell}
together with our main result.
We also give a model of
discrete flow (Section~\ref{sec:model}) on the codimension-one faces
along facets of the complex whose behavior is captured by the group
structure. Finally, we outline (Section~\ref{sec:intersection}) an
alternative interpretation of the higher critical groups as discrete
analogues of the Chow groups of an algebraic variety.

The authors thank Andy Berget, Hailong Dao, Craig Huneke, Manoj Kummini,
Gregg Musiker, Igor Pak, Vic Reiner, and Ken Smith for numerous helpful discussions.

\section{Critical Groups of Graphs}
\label{sec:graphs}

\subsection{The chip-firing game}
We summarize the chip-firing game on a
graph, omitting the proofs.  For 
more details,
see, e.g., Biggs~\cite{Biggs}.

Let $G=(V,E)$ be a finite, simple\footnote{The chip-firing game and our ensuing results can easily be extended to
allow parallel edges; we assume that $G$ is simple for the sake of ease of exposition.}, connected, undirected graph, with
$V=[n] \cup q =\{1,2,\dots,n,q\}$ and $E=\{e_1,\dots,e_m\}$. The special vertex $q$ is called the \emph{bank}
(or ``root'' or ``government'').  Let $d_i$ be the
degree of vertex~$i$, i.e. the number of adjacent vertices.
The chip-firing game is a discrete dynamical system whose state is described by a \emph{configuration} vector~$\cc=(c_1,\dots,c_n)\in\Nn^n$.
Each $c_i$ is a nonnegative integer that we think of as the number of ``chips'' belonging to
vertex~$i$.  (Note that the number $c_q$ of chips belonging to the bank~$q$ is not part of the data of a configuration.)

Each non-root vertex is generous (it likes to donate
chips to its neighbors), egalitarian (it likes all
its neighbors equally), and prudent (it does not want to go into debt).
Specifically, a vertex~$v_i$ is called \emph{ready} in a configuration $\cc$ if $c_i\geq
d_i$.  If a vertex is ready, it can \emph{fire} by giving one chip to each
of its neighbors.  Unlike the other vertices, the bank is a miser.  As long as other vertices
are firing, the bank does not fire, but just collects chips.

  As more
and more chips accumulate at the bank, the game eventually reaches a configuration
in which no non-bank vertex can fire.  Such a
configuration is called \emph{stable}.  At this point, the bank
finally fires, giving one chip to each of its neighbors.  Unlike the
other vertices, the bank is allowed to go into debt: that is, we do not require
that $c_q\geq d_q$ for the bank to be able to fire.

Denote by $\cc(x_1,\dots,x_r)$ the configuration obtained from
$\cc$ by firing the vertices $x_1,\dots,x_r$ in order.
This sequence (which may contain repetitions) is called a
\emph{firing sequence} for $\cc$ if every firing is permissible:
that is, for each $i$, either $x_i\neq q$ is ready to fire in the configuration $\cc(x_1,\dots,x_{i-1})$,
or else $x_i=q$ and $\cc(x_1,\dots,x_{i-1})$ is stable.  A
configuration $\cc$ is called \emph{recurrent} if there
is a nontrivial firing sequence $X$ such that $\cc(X)=\cc$.

A configuration is called \emph{critical} if it is both stable and recurrent.  For every
starting configuration $\cc$, there is a uniquely determined critical configuration
$[\cc]$ that can be reached from~$\cc$ by some firing sequence \cite[Thm.~3.8]{Biggs}.  The \emph{critical
group} $K(G)$ is defined as the set of these critical configurations, with group law given by
  $[\cc]+[\cc']=[\cc+\cc']$,
where the right-hand addition is componentwise addition of vectors.

The \emph{abelian sandpile model} was first introduced in \cite{Dhar} as an illustration of ``self-organized criticality'';
an excellent recent exposition is \cite{LP}.
Here, grains of sand (analogous to chips) are piled at each
vertex, and an additional grain of sand is added to a (typically
randomly chosen) pile.  If the pile reaches some predetermined
size (for instance, the degree of that vertex), then it \emph{topples}
by giving one grain of sand to each of its neighbors, which can
then topple in turn, and so on.  This sequence of topplings is called an \emph{avalanche}
and the associated operator on states of the system is called an
\emph{avalanche operator}.  (One can show that the avalanche operator
does not depend on the order in which vertices topple; this
is the reason for the use of the term ``abelian''.)  The sandpile model
itself is the random walk on the stable configurations, and the
critical group is the group generated by the avalanche operators.

The critical group can also be viewed as a discrete
analogue of the Picard group of an algebraic curve.  This point of
view goes back at least as far as the work of Lorenzini~\cite{Lor1,Lor2}
and was developed, using the language of divisors, by Bacher,
de~la~Harpe, and Nagnibeda \cite{BHN} (who noted that their ``setting has a straightforward generalization to higher dimensional objects'').  It appears in diverse combinatorial contexts including elliptic
curves over finite fields (Musiker~\cite{Musiker}), linear systems on
tropical curves (Haase, Musker and Yu~\cite{HMY}), and
Riemann-Roch theory for graphs (Baker and Norine~\cite{BN}).

\subsection{The algebraic viewpoint}
The critical group can be defined algebraically in terms of the Laplacian matrix.
\begin{definition}
Let $G$ be a finite, simple, connected, undirected graph with vertices $\{1,\dots,n,q\}$.
The {Laplacian matrix} of $G$ is the symmetric matrix $L$ (or, equivalently, linear self-adjoint operator)
whose rows and columns
are indexed by the vertices of $G$, with entries
  \begin{displaymath}
    \ell_{ij} = \begin{cases}
    d_i &\text{ if } i=j,\\
    -1  &\text{ if } ij\in E,\\
     0  &\text{ otherwise.}
  \end{cases}
  \end{displaymath}
\end{definition}

Firing vertex~$i$ in the chip-firing game is equivalent to subtracting the $i^{th}$ column of the Laplacian (ignoring the entry indexed by~$q$) from the configuration vector~$\cc$.
Equivalently, if $\cc'=\cc(x_1,\dots,x_r)$, then the configurations $\cc$ and $\cc'$
represent the same element of the cokernel
of the Laplacian (that is, the quotient of $\Zz^{n+1}$ by the column space of~$L$).

It is immediate from the definition of $L$ that
$L(\onev)=\zerov$, where $\onev$ and $\zerov$ denote the all-ones and all-zeros vectors in $\Nn^{n+1}$.  Moreover,
it is not difficult to show that $\rank L=|V|-1=n$.
In terms of homological algebra, we have a chain complex
  \begin{equation} \label{graph-chain}
  \Zz^{n+1} \xrightarrow{L} \Zz^{n+1} \xrightarrow{S} \Zz \to 0
  \end{equation}
where $S(\cc) = \cc\cdot\onev = c_q+c_1+\cdots+c_n$.  The equation $L(\onev)=\zerov$
says that $\ker(S)\supseteq\im(L)$.  Moreover, $\rank L=n=\rank\ker S$,
so the abelian group $\ker(S) / \im(L)$ is finite.

\begin{definition}
The \emph{critical group} of a graph $G$ is $K(G) = \ker(S)/\im(L)$.
\end{definition}

This definition of the critical group is equivalent to that in
terms of the chip-firing game \cite[Thm.~4.2]{Biggs}.  
The order of the critical group is the determinant of the
\emph{reduced Laplacian} formed by removing the row and column indexed by~$q$
\cite[Thm.~6.2]{Biggs}.
By the matrix-tree theorem, this is the number of spanning trees.  As we will see,
the algebraic description provides a natural framework for generalizing
the critical group.

\section{The Critical Groups of a Simplicial Complex}
We assume familiarity with the basic algebraic topology of simplicial complexes; see, e.g., Hatcher~\cite{Hatcher}.
Let $\Delta$ be a $d$-dimensional simplicial complex.  For $-1\leq i\leq d$, let $C_i(\Delta;\Zz)$ be the $i^{th}$
simplicial chain group of $\Delta$.  We denote the
simplicial boundary and coboundary maps respectively by
  \begin{align*}
  \bd_{\Delta,i}  &\;:\; C_i(\Delta;\Zz) \to C_{i-1}(\Delta;\Zz),\\
  \cbd_{\Delta,i} &\;:\; C_{i-1}(\Delta;\Zz) \to C_i(\Delta;\Zz),
  \end{align*}
where we have identified cochains with chains via the natural
inner product.  We will abbreviate the subscripts in
the notation for boundaries and coboundaries whenever no ambiguity can arise.

Let $-1\leq i\leq d$.  The \emph{$i$-dimensional combinatorial Laplacian}\footnote{\label{L-notation-note}%
In other settings, our Laplacian might be referred to as the ``up-down'' Laplacian, $L^{\rm ud}$.  The $i^{th}$ \emph{down-up Laplacian} is $\Ldu_i=\cbd_i \bd_i$,
and the $i^{th}$ \emph{total Laplacian} is $\Ltot_i = \Lud_i+\Ldu_i$.
We adopt the notation we do since, except for
one application (Remark~\ref{Molly} below), we only need the up-down Laplacian.
}
of $\Delta$
is the operator
\begin{displaymath}
\Lud_{\Delta,i}  = \bd_{i+1}\cbd_{i+1}\st C_i(\Delta;\Zz)\to C_i(\Delta;\Zz).
\end{displaymath}
Combinatorial Laplacian operators seem to have first appeared in the
work of Eckmann~\cite{Eckmann} on finite dimensional Hodge theory. As the name
suggests, they are discrete versions of the Laplacian operators on differential
forms on a Riemannian manifold. In fact, Dodziuk and Patodi~\cite{DP}
showed that for suitably nice
triangulations of a manifold, the eigenvalues of the discrete Laplacian converge in an appropriate sense to those
of the usual continuous Laplacian.
For one-dimensional complexes, i.e., graphs, the combinatorial Laplacian is
just the usual Laplacian matrix $L=D-A$, where $D$ is the diagonal
matrix of vertex degrees and $A$ is the (symmetric) adjacency matrix.

In analogy to the chain complex of \eqref{graph-chain}, we have the chain complex
\begin{displaymath}
C_i(\Delta;\Zz) \xrightarrow{L} C_i(\Delta;\Zz) \xrightarrow{\bd_i} C_{i-1}(\Delta;\Zz),
\end{displaymath}
where $L = \Lud_{\Delta,i}$.  (This is a chain complex because $\bdo_iL=\bd_i\bd_{i+1}\cbd_{i+1}=0$.)
We are now ready to make our main definition.

\begin{definition}
The $i$-dimensional critical group of $\Delta$ is 
\begin{displaymath}
K_i(\Delta) := \ker \bd_i / \im L = \ker \bd_i / \im(\bd_{i+1}\cbd_{i+1}).
\end{displaymath}
\end{definition}
Note that $K_0(\Delta)$ is precisely the critical group of the 1-skeleton of $\Delta$.

\subsection{Simplicial spanning trees}

Our results about critical groups rely on the theory of simplicial and
cellular spanning trees developed in~\cite{DKM}, based on earlier work
of Bolker~\cite{Bolker} and Kalai~\cite{Kalai}.  Here we briefly
review the definitions and basic properties, including the
higher-dimensional analogues of Kirchhoff's matrix-tree theorem.  For
simplicity, we present the theory for simplicial complexes, the case
of primary interest in combinatorics.  Nevertheless, the definitions
of spanning trees, their enumeration using a generalized matrix-tree
theorem, and the definition and main result about critical groups are
all valid in the more general setting of regular CW-complexes~\cite{DKM2}.

In order to define simplicial spanning trees, we first fix
some notation concerning simplicial complexes and algebraic topology.
The symbol~$\Delta_i$ will denote the set of cells of dimension~$i$.
The \emph{$i$-dimensional skeleton}~$\Delta_{(i)}$ of a simplicial complex~$\Delta$
is the subcomplex consisting of all cells of dimension~$\leq i$.  A complex is \emph{pure} if all maximal cells
have the same dimension.  The $i^{th}$ reduced homology group of~$\Delta$ with coefficients in a ring $R$
is denoted $\HH_i(\Delta;R)$.  The \emph{Betti numbers} of $\Delta$ are $\beta_i(\Delta)=\dim_\Qq\HH_i(\Delta;\Qq)$.
The \emph{$f$-vector} is $f(\Delta)=(f_{-1}(\Delta),f_0(\Delta),\dots)$, where $f_i(\Delta)$ is the number of
faces of dimension~$i$.

\begin{definition} \label{SST}
Let $\Delta$ be a pure $d$-dimensional simplicial complex, and let $\Upsilon\subseteq\Delta$ be a subcomplex
such that $\Upsilon_{(d-1)}=\Delta_{(d-1)}$.  We say that
$\Upsilon$ is a \emph{(simplicial) spanning tree} of $\Delta$ if the
following three conditions hold:
\begin{enumerate}
\item $\HH_d(\Upsilon;\Zz) = 0$;
\item $ \HH_{d-1}(\Upsilon;\Qq)=0$ (equivalently,  $|\HH_{d-1}(\Upsilon;\Zz)| < \infty$);
\item $f_d(\Upsilon) = f_d(\Delta)-\beta_d(\Delta)+\beta_{d-1}(\Delta)$.
\end{enumerate}
More generally, an \emph{$i$-dimensional spanning tree} of $\Delta$ is a spanning tree of the $i$-dimensional
skeleton of $\Delta$.
\end{definition}

In the case $d=1$ (that is, $\Delta$ is a graph), we recover the usual
definition of a spanning tree: the three conditions above say
respectively that~$\Upsilon$ is acyclic, connected, and has one more
vertex than edge.  Meanwhile, the 0-dimensional spanning trees
of~$\Delta$ are its vertices (more precisely, the subcomplexes
of~$\Delta$ with a single vertex), which are precisely the connected,
acyclic subcomplexes of~$\Delta_{(0)}$.

Just as in the graphical case, any two of the conditions of Definition~\ref{SST}
imply the third~\cite[Prop~3.5]{DKM}.
In order for~$\Delta$ to have a $d$-dimensional spanning tree, it is necessary and sufficient
that $\HH_i(\Delta;\Qq)=0$ for all $i<d$; such a complex is called \emph{acyclic in positive codimension},
or APC.  Note that a graph is APC if and only if it is connected.

\begin{example}\label{bipyr-trees}
Consider the \emph{equatorial bipyramid}: the two-dimensional simplicial
complex $B$ with vertices $[5]$ and facets $123, 124, 125, 134, 135,
234, 235$. A geometric realization of~$B$ is shown in
Figure~\ref{bipyramid-figure}.  A 2-SST of $B$ can be constructed by
removing two facets $F, F'$, provided that $F \cap F'$ contains
neither of the vertices $4, 5$. A simple count shows that there are 15
such pairs $F,F'$, so $B$ has 15 two-dimensional spanning trees.
\end{example}

\begin{figure}[ht] 
\begin{center}
\resizebox{1.05in}{1.35in}{\includegraphics{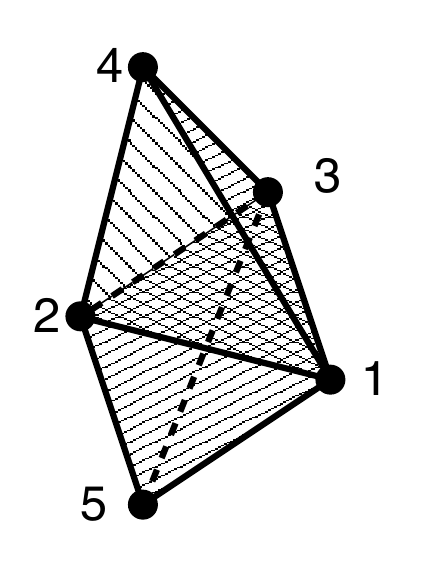}}
\end{center}
\caption{The equatorial bipyramid $B$.\label{bipyramid-figure}}
\end{figure}

A phenomenon arising only in dimension $d>1$ is that
spanning trees may have torsion: that is, $\HH_{d-1}(\Upsilon;\Zz)$ can be finite but nontrivial.  For example,
the 2-dimensional skeleton of a 6-vertex simplex has (several) spanning trees $\Upsilon$ that are homeomorphic
to the real projective plane, and in particular have $\HH_1(\Upsilon;\Zz)\isom\Zz/2\Zz$.  This cannot happen
in dimension~1 (i.e., for graphs), in which every spanning tree is a contractible topological space.
This torsion directly affects tree enumeration in higher dimension; see Section~\ref{order-section}.

\subsection{The main theorem}
Our main result gives an explicit form for the critical group $K_i(\Delta)$ in terms of a reduced Laplacian
matrix.  This reduced form is both more convenient for
computing examples, and gives a direct connection with the simplicial and cellular generalizations
of the matrix-tree theorem \cite{DKM,DKM2}.  For a general reference on the homological algebra
we will need, see, e.g., Lang \cite{Lang}.

Let~$\Delta$ be a pure, $d$-dimensional, APC simplicial complex, and fix $i<d$.
Let~$\Upsilon$ be an $i$-dimensional spanning tree of~$\Delta_{(i)}$,
and let $\Theta=\Delta_i \sm \Upsilon$ (the set of $i$-dimensional faces
of~$\Delta$ \emph{not} in~$\Upsilon$).  Let $\tilde L$
denote the reduced Laplacian obtained from~$L$ by removing the rows
and columns corresponding to~$\Upsilon$
(equivalently, by restricting~$L$ to the rows and columns corresponding to~$\Theta$).  

\begin{theorem}
\label{main-thm}
Suppose that $\HH_{i-1}(\Upsilon;\Zz)=0$.  Then
\begin{displaymath}
K_i(\Delta) \isom \Zz^\Theta / \im \tilde L.
\end{displaymath}
\end{theorem}

\begin{proof} 
We will construct a commutative diagram
\begin{equation}
\label{snake}
\xymatrix{
0 \ar[r] & \im L \ar[r]\ar[d]_f & \ker\bdo_{\Delta,i}\ar[r]\ar[d]_g & K_i(\Delta) \ar[d]_h\ar[r] & 0\\
0 \ar[r] & \im\rL \ar[r] & \Zz^\Theta\ar[r] & \Zz^\Theta / \im \tilde L\ar[r] & 0
}
\end{equation}
where the rows are short exact sequences with the natural inclusions and
quotient maps.
The map~$f$ is defined by $f(L\theta)=\rL\theta$ for all $\theta\in\Theta$ (which we will show is an isomorphism in Claim \ref{claim-three}), and
the map~$g$ is defined by $g(\hat\theta)=\theta$ (which we will show is an isomorphism in Claim \ref{claim-four}).
In Claim~\ref{claim-five}, we will show that
$f(\gamma)=g(\gamma)$ for all $\gamma\in\im L$, so 
the left-hand square commutes.  Having proven these facts, the map~$h$ is well-defined by
a diagram-chase, and it is an isomorphism by the snake lemma.
We organize the proof into a series of claims.

\begin{claim}\label{claim-one}
$\im\bd_{\Delta,i}=\im\bd_{\Upsilon,i}$ as $\Zz$-modules.
\end{claim}

Indeed, we have
$\im\bd_{\Upsilon,i}\subseteq\im\bd_{\Delta,i}\subseteq\ker\bd_{\Delta,i-1}=\ker\bd_{\Upsilon,i-1}$
(the last equality because $\Upsilon_{(i-1)}=\Delta_{(i-1)}$),
so there is a short exact sequence of $\Zz$-modules
$$
0 \to
\frac{\im\bd_{\Delta,i}}{\im\bd_{\Upsilon,i}} \to
\frac{\ker\bd_{\Upsilon,i-1}}{\im\bd_{\Upsilon,i}} \to
\frac{\ker\bd_{\Delta,i-1}}{\im\bd_{\Delta,i}} \to 0.
$$
Since $\Upsilon$ is a torsion-free spanning tree,
the middle term $\HH_{i-1}(\Upsilon;\Zz)$ is zero.
Therefore, the first term is zero
as well, proving Claim~\ref{claim-one}.

\begin{claim}
\label{subclaim}
$\coker\bd_{\Upsilon,i}$ is a free $\Zz$-module.
\end{claim}

We will use some of the basic theory of projective modules \cite[pp.~137--139]{Lang}.
The image of $\bd_{\Upsilon,i}$ is a submodule of $C_{i-1}(\Upsilon;\Zz)$, so it is free,
hence projective.  Therefore, the short exact sequence
$$0 \to \ker\bd_{\Upsilon,i} \to C_i(\Upsilon;\Zz) \xrightarrow{\bd_{\Upsilon,i}} \im \bd_{\Upsilon,i} \to 0$$
is split: that is, $C_{i}(\Upsilon;\Zz) = \ker\bd_{\Upsilon,i} \oplus F$,
where $F$ is a free $\Zz$-module.  On the other hand,
$\im\bd_{\Upsilon,i}\subseteq\ker\bd_{\Upsilon,i-1}$, so
$$
\coker \bd_{\Upsilon,i}
= \frac{C_{i-1}(\Upsilon;\Zz)}{\im \bd_{\Upsilon,i}}
= \frac{\ker\bd_{\Upsilon,i-1}\oplus F}{\im \bd_{\Upsilon,i}}
= \frac{\ker\bd_{\Upsilon,i-1}}{\im \bd_{\Upsilon,i}} \oplus F
= \HH_{i-1}(\Upsilon;\Zz) \oplus F$$
and $\HH_{i-1}(\Upsilon;\Zz)=0$ by hypothesis, proving Claim~\ref{subclaim}.

\begin{claim}\label{claim-two}
The coboundary map $\cbd_{\Upsilon,i}\st C_{i-1}(\Upsilon;\Zz) \to C_i(\Upsilon;\Zz)$ is surjective.
\end{claim}
By the basic theory of finitely generated abelian groups, we may write
\begin{equation} \label{not-smith}
\bd_{\Upsilon,i}
= P \left[\begin{array}{c|c} D&0\\ \hline 0&0\end{array}\right]Q
\end{equation}
where $P\in GL_{f_{i-1}}(\Zz)$,  $Q\in GL_{f_i}(\Zz)$, and $D$ is a diagonal matrix whose entries are the cyclic
summands of the torsion submodule of $\coker\bd_{\Upsilon,i}$.  In fact, these entries are all~1 by
Claim~\ref{subclaim}.  Moreover, the columns of $\bd_{\Upsilon,i}$ are linearly independent over~$\Qq$
and over~$\Zz$ because $\Upsilon$ is a simplicial tree, so in fact there are no zero columns in \eqref{not-smith}.
Therefore
$$\bd_{\Upsilon,i}
= P \left[\begin{array}{c} I\\ \hline 0\end{array}\right] Q
= P \left[\begin{array}{c} Q\\ \hline 0\end{array}\right]$$
and transposing yields
$$\cbd_{\Upsilon,i}
= \left[\begin{array}{c|c} Q^T & 0\end{array}\right] P^T$$
and so 
\begin{align*}
\im\cbd_{\Upsilon,i}&=\im\left[\begin{array}{c|c} Q^T & 0\end{array}\right]
&& \text{(because $P$, hence $P^T$, is invertible over $\Zz$)}\\
&=C_i(\Upsilon;\Zz)
&& \text{(because $Q$, hence $Q^T$, is invertible over $\Zz$)}.
\end{align*}
We have proved Claim~\ref{claim-two}.

\begin{claim}\label{claim-three}
$L(C_i(\Delta;\Zz)) = L(C_i(\Theta;\Zz))$.
\end{claim}

Choose an arbitrary chain $\gamma\in C_i(\Upsilon;\Zz)$.  By Claim~\ref{claim-two}, there is a chain $\eta\in C_{i-1}(\Upsilon;\Zz)$
such that $\cbd_{\Upsilon,i}(\eta) = \gamma$.  On the other hand,
$\cbd_{\Delta,i}(\eta) = \cbd_{\Upsilon,i}(\eta) - \theta$ for some chain~$\theta\in C_i(\Theta;\Zz)$.  Hence
\begin{displaymath}
L(\gamma)-L(\theta)
~=~ L(\gamma-\theta)
~=~ L(\cbd_{\Upsilon,i}(\eta)-\theta)
~=~ L\cbd_{\Delta,i}(\eta)
~=~ \bd_{i+1} \cbd_{i+1} \cbd_i(\eta)
~=~ 0
\end{displaymath}
and so $L(\gamma) = L(\theta)$.  In particular, $L(\gamma)\in L(C_i(\Theta;\Zz))$,
which proves Claim~\ref{claim-three}.

\medskip

Observe that
$$L ~=~ \left[\begin{array}{c|c}
L(C_i(\Upsilon;\Zz)) & L(C_i(\Theta;\Zz))
\end{array}\right]
~=~ \left[\begin{array}{c|c}
L(C_i(\Upsilon;\Zz)) & \begin{array}{c} *\pad\\\hline\pad\rL \end{array}
\end{array}\right].
$$
Thus Claim~\ref{claim-three} says that the $\Theta$-columns span the full column space of~$L$.
Since $L$ is a symmetric matrix, this statement remains true if we replace ``column''
with ``row''.  In particular, there is an isomorphism $f\st L(C_i(\Theta;\Zz))\to\rL(C_i(\Theta;\Zz))$
given by deleting the $\Upsilon$-rows; that is, $f(L\theta)=\rL\theta$.

By Claim~\ref{claim-one}, for each chain $\theta\in C_i(\Theta;\Zz)$, we can write
$$\bdo_{\Delta,i}(\theta)=\sum_{\sigma\in\Upsilon_i}c^{\phantom{*}}_{\sigma\theta}\bdo_{\Delta,i}(\sigma)$$
with $c^{\phantom{*}}_{\sigma\theta}\in\Zz$.  Therefore, the chain
$$\hat \theta=\theta-\sum_{\sigma\in\Upsilon_i}c^{\phantom{*}}_{\sigma\theta}\sigma$$
lies in $X:=\ker\bd_{\Delta,i}$.

\begin{claim}\label{claim-four}
The set $\{\hat\theta\st\theta\in\Theta\}$ is a $\Zz$-module basis for $X$.
\end{claim}

Indeed, for any $\gamma=\sum_{\sigma\in\Delta_i} a_\sigma\sigma\in X$, let
$$
\gamma' ~=~ \sum_{\sigma\in\Delta_i} a_\sigma\sigma ~-~ \sum_{\sigma\in\Theta} a_\sigma\hat\sigma
~=~ \sum_{\sigma\in\Upsilon_i} a_\sigma\sigma ~+~ \sum_{\sigma\in\Theta} a_\sigma(\sigma-\hat\sigma).
$$
By the previous observation, we have $\gamma'\in X\cap
C_i(\Upsilon;\Zz)=\HH_i(\Upsilon;\Zz)$.  On the other hand,
$\HH_i(\Upsilon;\Zz)=0$ (because $\Upsilon$ is an $i$-dimensional
simplicial tree), so in fact $\gamma'=0$.  Therefore, $\gamma =
\sum_{\sigma\in\Theta} a_\sigma \hat\sigma$, proving Claim~\ref{claim-four}.

\begin{claim}\label{claim-five}
Suppose that the action of the reduced Laplacian $\rL$ on $C_i(\Theta;\Zz)$ is given by
$$\rL \theta = \sum_{\sigma\in\Theta} \ell_{\theta\sigma} \sigma$$
for $\theta\in\Theta$.  Then
$L\theta=\sum_{\sigma\in\Theta}\ell_{\sigma\theta}\hat\sigma$.
\end{claim}

Indeed, the chain $L\theta-\sum_{\sigma\in\Theta}\ell_{\sigma\theta}\hat\sigma$ belongs both to~$X$ and to~$C_i(\Upsilon;\Zz)$, so it must be zero
(as in the proof of Claim~\ref{claim-four}, because $\Upsilon$ is an $i$-dimensional simplicial tree), establishing Claim~\ref{claim-five}
and completing the proof
\end{proof}

\begin{remark}
Claim~\ref{claim-one} holds for
any subcomplex $\Upsilon\subseteq\Delta$ \emph{containing}
a torsion-free spanning tree or, more interestingly, if $\Upsilon$
is ``torsion-minimal'', i.e., if $\HH_{i-1}(\Upsilon;\Zz) = \HH_{i-1}(\Delta;\Zz)$.  The remainder of the proof requires $\Upsilon$ to be
torsion-free, but it should be possible to extend these methods
to the case that it is 
torsion-minimal.

Not every APC complex need have a torsion-minimal spanning tree.  For instance, for~$k\in\Nn$, let $M_k=M(\Zz/k\Zz,2)$ be the Moore space
\cite[p.~143]{Hatcher} obtained by attaching a 2-cell to the circle $S^1$ by a map of degree~$k$; let $X$ be the complex obtained
by identifying the 1-skeletons of~$M_2$ and~$M_3$ (thus, $X$ is a CW-complex with two 2-cells, one 1-cell, and one 0-cell); and
let~$\Delta$ be a simplicial triangulation of~$X$
(in particular, note that $\Delta$ is a regular CW-complex with the same homology as $X$).
Then $\HH_2(\Delta;\Zz)\isom\Zz$ and $\HH_i(\Delta;\Zz)=0$ for $i<2$.  On the other hand, the simplicial spanning trees of~$\Delta$ are precisely the subcomplexes~$\Upsilon$ obtained by deleting a single facet~$\sigma$
(just as though $\Upsilon$ were a simplicial sphere, which it certainly is not---see
Remark~\ref{sphere-count} and subsequently), and each such 
$\Upsilon$ has $\HH_1(\Upsilon;\Zz)\isom\Zz/3\Zz$ or $\Zz/2\Zz$, according as $\sigma$ is a face of $M_2$ or $M_3$.
(We thank Vic Reiner for providing this example.)	
\end{remark}

\begin{example}\label{bipyr-thm}
We return to the bipyramid $B$ from Example \ref{bipyr-trees} to illustrate Theorem \ref{main-thm}.  We must first pick a 1-dimensional spanning tree $\Upsilon$; we take $\Upsilon$ to be the spanning tree with edges $12, 13, 14, 15$.  (In general, we must also make sure $\Upsilon$ is torsion-free, but this is always true for 1-dimensional trees.)  Let $L = \Lud_{B,1}\st C_1(B;\Zz) \rightarrow C_1(B;\Zz)$ be the full Laplacian; note that $L$ is a $9\x 9$ matrix whose rows and columns are indexed by the edges of $B$.  The reduced Laplacian $\tilde L$ is formed by removing the rows and columns indexed by the edges of $\Upsilon$:
\begin{displaymath}
\tilde L ~=~
\bordermatrix{
     & 23 & 24 & 25 & 34 & 35 \cr
23 & 3 & -1 & -1 & 1 & 1 \cr
24 & -1 &2 & 0 & -1 & 0 \cr
25 & -1 & 0 & 2 & 0 & -1 \cr
34 & 1 & -1 & 0 & 2 & 0 \cr
35 & 1 & 0 & -1 & 0 & 2
}.
\end{displaymath}
The critical group $K_1(B)$ is the cokernel of this matrix, i.e., $K_1(B)\isom\Zz^5/\im \tilde L$.  Since $\tilde L$ has full rank, it follows that $K_1(B)$ is finite; its order is $\det(\tilde L) = 15$.
\end{example}

\section{The Order of the Critical Group}\label{order-section}

The matrix-tree theorem implies that the order of the critical group of a graph equals the number of
spanning trees.  In this section, we explain how this equality carries over to the higher-dimensional setting.

As before, let $\Delta$ be a pure $d$-dimensional simplicial complex.
Let $\SST_i(\Delta)$ denote the set of all $i$-dimensional spanning trees of $\Delta$
(that is, spanning trees of the $i$-dimensional skeleton~$\Delta_{(i)}$).
Define 
\begin{align*}
   \tau_i  &= \sum_{\Upsilon\in\SST_i(\Delta)} |\HH_{i-1}(\Upsilon;\Zz)|^2,\\
   \pi_i   &= \text{ product of all nonzero eigenvalues of $L_{\Delta, i-1}$}.
\end{align*}
The following formulas relate the tree enumerators $\tau_i$ to the linear-algebraic
invariants $\pi_i$.

\begin{theorem}[The simplicial matrix-tree theorem]
\cite[Thm.~1.3]{DKM}
\label{thm:SMTT}
For all $i\leq d$, we have
\begin{displaymath}
\pi_i = \frac{\tau_i \tau_{i-1}}{|\HH_{i-2}(\Delta;\Zz)|^2}.
\end{displaymath}
Moreover, if  $\Upsilon$ is any spanning tree of $\Delta_{(i-1)}$,
then
\begin{displaymath}
\tau_i = \frac{|\HH_{i-2}(\Delta;\Zz)|^2}{|\HH_{i-2}(\Upsilon;\Zz)|^2} \det \tilde{L},
\end{displaymath}
where $\tilde{L}$ is the reduced Laplacian formed by removing the rows and columns corresponding to $\Upsilon$.
\end{theorem}
Recall that when $d=1$, the number $\tau_1(\Delta)$ is simply the number of spanning
trees of the graph $\Delta$, and $\tau_0(\Delta)$ is the number of vertices (i.e., 0-dimensional
spanning trees).  Therefore, the formulas above specialize to the classical matrix-tree theorem. 

\begin{corollary}
\label{count-corollary}
Let $i<d$.  Suppose that $\HH_{i-1}(\Delta;\Zz)=0$ and that $\Delta$ has an $i$-dimensional spanning tree
$\Upsilon$ such that $\HH_{i-1}(\Upsilon;\Zz)=0$.  Then
the order of the $i$-dimensional critical group is the torsion-weighted number of $(i+1)$-dimensional
spanning trees, i.e.,
\begin{displaymath}
|K_i(\Delta)|=\tau_{i+1}.
\end{displaymath}
\end{corollary}

\begin{example}\label{bipyr-count}
Returning again to the bipyramid $B$, recall that 15 is both the number of its spanning trees (Example \ref{bipyr-trees}) and the order of its 1-dimensional critical group (Example \ref{bipyr-thm}), in each case because $\det \tilde L = 15$.
\end{example}

Another formula for the orders of the critical groups of~$\Delta$ is
as follows.  For $0\leq j\leq d$, denote by $\pi_j$ the product of the
nonzero eigenvalues of the Laplacian $L_{j-1}^{ud}=\bd_j\cbd_j$.  Then
Corollary~\ref{count-corollary}, together with
\cite[Corollary~2.10]{DKM2}, implies the following formula for
$|K_i(\Delta)|$ as an alternating product:

\begin{corollary} \label{alt-product}
Under the conditions of Corollary ~\ref{count-corollary}, for every $i\leq d$, we have
\begin{displaymath}
|K_i(\Delta)| = \prod_{j=0}^i \pi_j^{(-1)^{i-j}}.
\end{displaymath}
\end{corollary}

The condition that $\Delta$ and $\Upsilon$ be torsion-free is not too restrictive, in the sense that many simplicial complexes of interest in combinatorics
(for instance, all shellable complexes) are torsion-free and have torsion-free spanning trees. 
\begin{remark}
\label{sphere-count}
When every spanning tree of $\Delta$ is torsion-free, the order of the
critical group is exactly the number of spanning trees.  This is a
strong condition on $\Delta$, but it does hold for some complexes ---
notably for simplicial spheres, whose spanning trees are exactly the
(contractible) subcomplexes obtained by deleting a single facet.  Thus
a given explicit bijection between spanning trees and elements of the
critical group amounts to an abelian group structure on the set of
facets of a simplicial sphere.
\end{remark}

Determining the structure of the critical group is not easy, even for very
special classes of graphs; see, e.g., \cite{Vic1,Vic2}.  One of the
first such results is due to Lorenzini \cite{Lor1,Lor2} and Merris
\cite[Example~1(1.4)]{Merris}, who independently noted that the critical
group of the cycle graph on $n$ vertices is $\Zz/n\Zz$, the cyclic group
on $n$ elements.  Simplicial spheres are the natural
generalizations of cycle graphs from a tree-enumeration point of view.
In fact, the theorem of Lorenzini and Merris carries over to
simplicial spheres, as we now show.

\begin{theorem}
\label{spheres}
Let $\Sigma$ be a $d$-dimensional simplicial sphere with $n$ facets.  Then
$K_{d-1}(\Sigma) \cong \Zz/n\Zz$. 
\end{theorem}

\begin{proof}
Let $K=K_{d-1}(\Sigma)$.  Remark \ref{sphere-count} implies that $|K|=n$,
so it is sufficient to show that it is cyclic.
In what follows, we use the standard
terms ``facets'' and ``ridges'' for faces of~$\Sigma$
of dimensions~$d$ and~$d-1$, respectively.

By definition, $K$ is generated by
$(d-1)$-dimensional cycles, that is, elements of $\ker\bdo_{d-1}$.  Since
$\HH_{d-1}(\Sigma;\Zz)=0$, all such cycles are in fact
$(d-1)$-dimensional boundaries of $d$-dimensional chains.  Therefore,
$K$ is generated by the boundaries of facets, modulo the image
of $L=\bdo_d\cbd_d$.
We now show that for any two facets $\sigma,\sigma'\in\Sigma$,
we have $\bdo_d\sigma\equiv\pm\bdo_d\sigma$ modulo $\im L$.
This will imply
that~$K$ can be generated by a single element as a $\Zz$-module.

Since $\Sigma$ is a sphere, it is in particular a pseudomanifold, so every ridge is
in the boundary of at most two facets \cite[p.~24]{Stanley}.
Consequently, if two facets $\sigma,\sigma'$ share a ridge $\rho$, then no other
facet contains $\rho$, and we have
$$0 \equiv \bdo\cbd(\rho) = \bdo(\pm \sigma \pm \sigma') = \pm(\bdo \sigma \pm \bdo \sigma')$$
(where $\equiv$ means ``equal modulo $\im L$'').
Hence $\bdo(\sigma)$ and $\bdo(\sigma')$ represent the same or opposite
elements of $K$.  Furthermore, the definition of pseudomanifold
guarantees that for any two facets $\sigma,\sigma'$, there is a
sequence of facets $\sigma = \sigma_0, \sigma_1, \ldots, \sigma_k = \sigma'$ such that each
$\sigma_j$ and $\sigma_{j+1}$ share a common ridge.  Therefore, by transitivity, the boundary
of any single facet generates~$K$, as desired.
\end{proof}

The condition that $\Sigma$ be a simplicial sphere can be relaxed:
in fact, the proof goes through for any $d$-dimensional
pseudomanifold $\Sigma$ such that $\HH_{d-1}(\Sigma; \Zz) = 0$.
On the other hand, if $\Sigma$ is APC in addition
to being a pseudomanifold (for example, certain lens spaces---see
\cite[p.~144] {Hatcher}), then
it has the rational homology type of either a sphere or a ball
(because $\HH_d(\Sigma;\Qq)$ is either $\Qq$ or 0;
see \cite[p.~24]{Stanley}).

\begin{remark}
\label{Molly}
Let $\Delta$ be the simplex on vertex set~$[n]$, and let
$k\leq n$.  Kalai \cite{Kalai}
proved that $\tau_k(\Delta)=n^{\binom{n-2}{k}}$ for every $n$ and $k$,
generalizing Cayley's formula $n^{n-2}$ for the number of
labeled trees on $n$ vertices.  Maxwell~\cite{Maxwell} studied
the skew-symmetric matrix
$$A = \left[\begin{array}{c}
\pad\rbd_{\Delta,k}\\ \hline \pad-\rcbd_{\Delta,k+1}
\end{array}\right]$$
where~$\rbd_{\Delta,k}$ denotes the reduced boundary map
obtained from the usual simplicial boundary $\bd_{\Delta,k}$ by deleting
the rows corresponding to $(k-1)$-faces containing vertex~1,
and~$\rcbd_{\Delta,k+1}$ is
obtained from~$\cbd_{\Delta,k+1}$ by deleting
the rows corresponding to $k$-faces \emph{not} containing vertex~1.
(Note that Maxwell and Kalai use the symbol $I^k_r(X)$
for what we call $\bd_{\Delta,k}$.)

In particular, Maxwell \cite[Prop.~5.4]{Maxwell} proved that
$$\coker A \isom (\Zz/n\Zz)^{\binom{n-2}{k}}.$$
The matrix~$A$ is not itself a Laplacian, but is closely related
to the Laplacians of $\Delta$.  Indeed,
Maxwell's result, together with ours, implies that
all critical groups of~$\Delta$ are direct sums of cyclic groups
of order~$n$, for the following reasons.  We have
$$
AA^T = -A^2 =
\left[\begin{array}{c}
\pad\rbd_{\Delta,k}\\ \hline -\pad\rcbd_{\Delta,k+1}
\end{array}\right]
\left[\begin{array}{c|c}
\rcbd_{\Delta,k} & -\rbd_{\Delta,k+1}
\end{array}\right]
=
\left[\begin{array}{c|c}
\pad\rL^{\textrm{ud}}_{k-1} & 0\\ \hline
0 & \pad\rL^{\textrm{du}}_{k+1}
\end{array}\right]
$$
where ``ud'' and ``du'' stand for ``up-down'' and ``down-up'' respectively
(see footnote~\ref{L-notation-note}).  Therefore
\begin{align*}
\coker(AA^T)
&\isom \coker(\rL^{\textrm{ud}}_{k-1}) \oplus \coker(\rL^{\textrm{du}}_{k+1})\\
&\isom \coker(\rL^{\textrm{ud}}_{k-1}) \oplus \coker(\rL^{\textrm{ud}}_{k})\\
&\isom K_{k-1}(\Delta)\oplus K_k(\Delta),
\end{align*}
where the second step follows from the general fact that~$MM^T$
and~$M^TM$ have the same multisets of nonzero eigenvalues
for any matrix~$M$, and the third step follows from
Theorem~\ref{main-thm}.  On the other hand, we have
$\coker(A A^T) = \coker(-A^2) = \coker(A^2) \isom(\coker A) \oplus (\coker A)$.
It follows from Maxwell's result that the $k$th critical group
of the $n$-vertex simplex is a direct sum of $\binom{n-2}{k}$ copies of~$\Zz/n\Zz$, as desired.
\end{remark}

\section{The Critical Group as a Model of Discrete Flow}
\label{sec:model}
In this section, we describe an interpretation of the critical group in terms of flow, analogous to the chip-firing game.  By definition of $K_i(\Delta)$, its elements may be represented as integer vectors $\cc=(c_F)_{F\in\Delta_i}$, modulo an equivalence relation given by the Laplacian.  These configurations are the analogues of the configurations of chips in the graph case ($i=0$).  
When $i=1$, it is natural to interpret $c_F$ as a flow along the edge~$F$, in the direction given by some predetermined orientation; a negative value on an edge corresponds to flow in the opposite direction.  More generally, if $F$ is an $i$-dimensional face, then we can interpret $c_F$ as a generalized \emph{$i$-flow}, again with the understanding that a negative value on a face means a $i$-flow in the opposite orientation.  For instance, 2-flow on a triangle represents circulation around the triangle, and a negative 2-flow means to switch between clockwise and counterclockwise.

When $i=1$, the condition $\cc\in\ker \bd_i$ means that flow neither accumulates nor depletes at any vertex; intuitively, matter is conserved.  In general, we call an $i$-flow \emph{conservative} if it lies in $\ker \bd_i$.  For instance, when $i=2$, the $\bd_i$ map converts 2-flow around a single triangle into 1-flow along the three edges of its boundary in the natural way; for a 2-flow on $\Delta$ to be conservative, the sum of the resulting 1-flows on each edge must cancel out, leaving no net flow along any edge.  In general, the sum of (the boundaries of) all the $i$-dimensional flows surrounding an $(i-1)$-dimensional face must cancel out along that face.

That the group $K_i(\Delta)$ is a quotient by the image of the Laplacian means that two configurations are equivalent if they differ by an integer linear combination of Laplacians applied to $i$-dimensional faces.  This is analogous to the chip-firing game, where configurations are equivalent when it is possible to get from one to the other by a series of chip-firings, each of which corresponds to adding a column vector of the Laplacian.
When $i=1$, it is easy to see that firing an edge~$e$ (adding the image of its Laplacian to a configuration) corresponds to diverting one unit of flow around each triangle containing~$e$ (see Example \ref{bipyr-flow}).  More generally, to fire an $i$-face $F$ means to divert one unit of $i$-flow from $F$ around each $(i+1)$-face containing~$F$.

By Theorem \ref{main-thm}, we may compute the critical group as $\Zz^\Theta$ modulo the image of the reduced Laplacian.  In principle, passing to the reduced Laplacian means ignoring the $i$-flow along each facet of an $i$-dimensional spanning tree $\Upsilon$.  In the graph case ($i=0$), this spanning tree is simply the bank vertex.  The higher-dimensional generalization of this statement is that the equivalence class of a configuration $\cc$ is determined
by the subvector $(c_F)_{F\in\Delta\sm\Upsilon}$.

A remaining open problem is to identify the higher-dimensional ``critical configurations'', i.e., a set of stable and recurrent
configurations that form a set of coset representatives for the critical group.
Recall that in the chip-firing game, when vertex $i$ fires, every vertex other than $i$ either gains a chip or stays unchanged.
Therefore, we can define stability simply by the condition $\cc_i<\deg(i)$ for every non-bank vertex $i$.  On the other hand, when
a higher-dimensional face fires, the flow along nearby faces can actually decrease.  Therefore, it is not as easy to define stability.
For instance, one could try to define stability by the condition that no face can fire without forcing some face (either itself or one of its
neighbors) into debt.  However, with this definition, there are some examples (such as the 2-skeleton of the tetrahedron) for which some of the cosets of the Laplacian admit more than one critical configuration.  Therefore, it is not clear how to choose a canonical set of coset representatives analogous
to the critical configurations of the graphic chip-firing game.

\begin{example}\label{bipyr-flow}
We return once again to the bipyramid $B$, and its 1-dimensional spanning tree $\Upsilon$ with edges $12, 13, 14, 15$.  If we pick 1-flows on $\Theta = \Delta_1 \sm \Upsilon$ as shown in Figure \ref{flow-exA}, it is easy to compute that we need 1-flows on $\Upsilon$ as shown in Figure \ref{flow-exB} to make the overall flow 1-conservative.  
Since Theorem \ref{main-thm} implies we can always pick flows on $\Upsilon$ to make the overall flow 1-conservative, we only show flows on $\Theta$ in subsequent diagrams.

If we fire edge 23, we get the configuration shown in Figure \ref{flow-exC}.  One unit of flow on edge 23 has been diverted across face 234 to edges 24 and 34, and another unit of flow has been diverted across face 235 to edges 25 and 35.  Note that the absolute value of flow on edge 25 has actually decreased, because of its orientation relative to edge 23.  If we subsequently fire edge 24, we get the configuration shown in Figure \ref{flow-exD}.  One unit of flow on edge 24 has been diverted across face 234 to edges 23 and 34, and another unit of flow has been diverted across face 124 to edges 12 and 14 (and out of the diagram of $\Theta$).

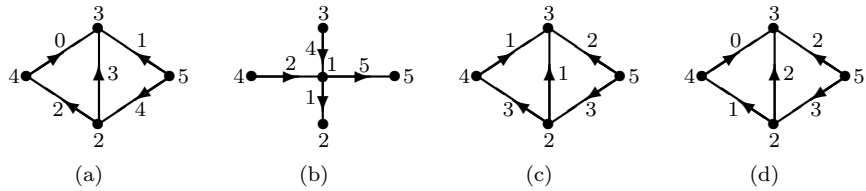
\begin{figure}[ht]
\begin{center}
\subfigure[\label{flow-exA}]{\setlength{\unitlength}{3947sp}%
\begingroup\makeatletter\ifx\SetFigFont\undefined%
\gdef\SetFigFont#1#2#3#4#5{%
  \reset@font\fontsize{#1}{#2pt}%
  \fontfamily{#3}\fontseries{#4}\fontshape{#5}%
  \selectfont}%
\fi\endgroup%
\begin{picture}(1095,922)(2574,-524)
\thicklines
\put(3158,239){\line( 3,-2){450}}
\put(3608,-61){\line(-3,-2){450}}
\put(3158,-361){\line(-3, 2){450}}
\put(2708,-61){\line( 3, 2){450}}
\put(3151,-361){\circle*{60}}
\put(3151,239){\circle*{60}}
\put(3601,-61){\circle*{60}}
\put(3158,239){\line( 0,-1){600}}
\put(3158,-361){\vector( 0, 1){375}}
\put(3608,-61){\vector(-3,-2){225}}
\put(3158,-361){\vector(-3, 2){225}}
\put(2708,-61){\vector( 3, 2){225}}
\put(3608,-61){\vector(-3, 2){225}}
\put(2701,-61){\circle*{60}}
\put(3654,-102){\makebox(0,0)[lb]{\smash{{\SetFigFont{8}{9.6}{\rmdefault}{\mddefault}{\updefault}5}}}}
\put(2865,-310){\makebox(0,0)[lb]{\smash{{\SetFigFont{8}{9.6}{\rmdefault}{\mddefault}{\updefault}2}}}}
\put(2878,120){\makebox(0,0)[lb]{\smash{{\SetFigFont{8}{9.6}{\rmdefault}{\mddefault}{\updefault}0
}}}}
\put(3393,124){\makebox(0,0)[lb]{\smash{{\SetFigFont{8}{9.6}{\rmdefault}{\mddefault}{\updefault}1}}}}
\put(3384,-315){\makebox(0,0)[lb]{\smash{{\SetFigFont{8}{9.6}{\rmdefault}{\mddefault}{\updefault}4}}}}
\put(3121,-509){\makebox(0,0)[lb]{\smash{{\SetFigFont{8}{9.6}{\rmdefault}{\mddefault}{\updefault}2}}}}
\put(3121,293){\makebox(0,0)[lb]{\smash{{\SetFigFont{8}{9.6}{\rmdefault}{\mddefault}{\updefault}3}}}}
\put(2589,-102){\makebox(0,0)[lb]{\smash{{\SetFigFont{8}{9.6}{\rmdefault}{\mddefault}{\updefault}4}}}}
\put(3211,-88){\makebox(0,0)[lb]{\smash{{\SetFigFont{8}{9.6}{\rmdefault}{\mddefault}{\updefault}3}}}}
\end{picture}%
}\hfil 
\subfigure[\label{flow-exB}]{\setlength{\unitlength}{3947sp}%
\begingroup\makeatletter\ifx\SetFigFont\undefined%
\gdef\SetFigFont#1#2#3#4#5{%
  \reset@font\fontsize{#1}{#2pt}%
  \fontfamily{#3}\fontseries{#4}\fontshape{#5}%
  \selectfont}%
\fi\endgroup%
\begin{picture}(1095,922)(4076,-521)
\thicklines
\put(4651,-361){\line( 0, 1){600}}
\put(4201,-61){\line( 1, 0){900}}
\put(4653,242){\circle*{60}}
\put(5103,-58){\circle*{60}}
\put(4651,-61){\circle*{60}}
\put(4651,239){\vector( 0,-1){225}}
\put(4201,-61){\vector( 1, 0){300}}
\put(4651,-61){\vector( 1, 0){300}}
\put(4651,-61){\vector( 0,-1){225}}
\put(4203,-58){\circle*{60}}
\put(4653,-358){\circle*{60}}
\put(4540, 74){\makebox(0,0)[lb]{\smash{{\SetFigFont{8}{9.6}{\rmdefault}{\mddefault}{\updefault}4}}}}
\put(4623,296){\makebox(0,0)[lb]{\smash{{\SetFigFont{8}{9.6}{\rmdefault}{\mddefault}{\updefault}3}}}}
\put(4091,-99){\makebox(0,0)[lb]{\smash{{\SetFigFont{8}{9.6}{\rmdefault}{\mddefault}{\updefault}4}}}}
\put(5156,-99){\makebox(0,0)[lb]{\smash{{\SetFigFont{8}{9.6}{\rmdefault}{\mddefault}{\updefault}5}}}}
\put(4675,-24){\makebox(0,0)[lb]{\smash{{\SetFigFont{8}{9.6}{\rmdefault}{\mddefault}{\updefault}1}}}}
\put(4877,-33){\makebox(0,0)[lb]{\smash{{\SetFigFont{8}{9.6}{\rmdefault}{\mddefault}{\updefault}5}}}}
\put(4540,-231){\makebox(0,0)[lb]{\smash{{\SetFigFont{8}{9.6}{\rmdefault}{\mddefault}{\updefault}1}}}}
\put(4415,-23){\makebox(0,0)[lb]{\smash{{\SetFigFont{8}{9.6}{\rmdefault}{\mddefault}{\updefault}2}}}}
\put(4623,-506){\makebox(0,0)[lb]{\smash{{\SetFigFont{8}{9.6}{\rmdefault}{\mddefault}{\updefault}2}}}}
\end{picture}%
}\hfil 
\subfigure[\label{flow-exC}]{\setlength{\unitlength}{3947sp}%
\begingroup\makeatletter\ifx\SetFigFont\undefined%
\gdef\SetFigFont#1#2#3#4#5{%
  \reset@font\fontsize{#1}{#2pt}%
  \fontfamily{#3}\fontseries{#4}\fontshape{#5}%
  \selectfont}%
\fi\endgroup%
\begin{picture}(1095,922)(5836,-524)
\thicklines
\put(6420,239){\line( 3,-2){450}}
\put(6870,-61){\line(-3,-2){450}}
\put(6420,-361){\line(-3, 2){450}}
\put(5970,-61){\line( 3, 2){450}}
\put(6413,-361){\circle*{60}}
\put(6413,239){\circle*{60}}
\put(6863,-61){\circle*{60}}
\put(6420,239){\line( 0,-1){600}}
\put(6420,-361){\vector( 0, 1){375}}
\put(6870,-61){\vector(-3,-2){225}}
\put(6420,-361){\vector(-3, 2){225}}
\put(5970,-61){\vector( 3, 2){225}}
\put(6870,-61){\vector(-3, 2){225}}
\put(5963,-61){\circle*{60}}
\put(6916,-102){\makebox(0,0)[lb]{\smash{{\SetFigFont{8}{9.6}{\rmdefault}{\mddefault}{\updefault}5}}}}
\put(6127,-310){\makebox(0,0)[lb]{\smash{{\SetFigFont{8}{9.6}{\rmdefault}{\mddefault}{\updefault}3}}}}
\put(6140,120){\makebox(0,0)[lb]{\smash{{\SetFigFont{8}{9.6}{\rmdefault}{\mddefault}{\updefault}1
}}}}
\put(6655,124){\makebox(0,0)[lb]{\smash{{\SetFigFont{8}{9.6}{\rmdefault}{\mddefault}{\updefault}2}}}}
\put(6646,-315){\makebox(0,0)[lb]{\smash{{\SetFigFont{8}{9.6}{\rmdefault}{\mddefault}{\updefault}3}}}}
\put(6383,-509){\makebox(0,0)[lb]{\smash{{\SetFigFont{8}{9.6}{\rmdefault}{\mddefault}{\updefault}2}}}}
\put(6383,293){\makebox(0,0)[lb]{\smash{{\SetFigFont{8}{9.6}{\rmdefault}{\mddefault}{\updefault}3}}}}
\put(5851,-102){\makebox(0,0)[lb]{\smash{{\SetFigFont{8}{9.6}{\rmdefault}{\mddefault}{\updefault}4}}}}
\put(6473,-88){\makebox(0,0)[lb]{\smash{{\SetFigFont{8}{9.6}{\rmdefault}{\mddefault}{\updefault}1}}}}
\end{picture}%
}\hfil 
\subfigure[\label{flow-exD}]{\setlength{\unitlength}{3947sp}%
\begingroup\makeatletter\ifx\SetFigFont\undefined%
\gdef\SetFigFont#1#2#3#4#5{%
  \reset@font\fontsize{#1}{#2pt}%
  \fontfamily{#3}\fontseries{#4}\fontshape{#5}%
  \selectfont}%
\fi\endgroup%
\begin{picture}(1095,922)(7486,-524)
\thicklines
\put(8070,239){\line( 3,-2){450}}
\put(8520,-61){\line(-3,-2){450}}
\put(8070,-361){\line(-3, 2){450}}
\put(7620,-61){\line( 3, 2){450}}
\put(8063,-361){\circle*{60}}
\put(8063,239){\circle*{60}}
\put(8513,-61){\circle*{60}}
\put(8070,239){\line( 0,-1){600}}
\put(8070,-361){\vector( 0, 1){375}}
\put(8520,-61){\vector(-3,-2){225}}
\put(8070,-361){\vector(-3, 2){225}}
\put(7620,-61){\vector( 3, 2){225}}
\put(8520,-61){\vector(-3, 2){225}}
\put(7613,-61){\circle*{60}}
\put(8566,-102){\makebox(0,0)[lb]{\smash{{\SetFigFont{8}{9.6}{\rmdefault}{\mddefault}{\updefault}5}}}}
\put(7777,-310){\makebox(0,0)[lb]{\smash{{\SetFigFont{8}{9.6}{\rmdefault}{\mddefault}{\updefault}1}}}}
\put(7790,120){\makebox(0,0)[lb]{\smash{{\SetFigFont{8}{9.6}{\rmdefault}{\mddefault}{\updefault}0}}}}
\put(8305,124){\makebox(0,0)[lb]{\smash{{\SetFigFont{8}{9.6}{\rmdefault}{\mddefault}{\updefault}2}}}}
\put(8296,-315){\makebox(0,0)[lb]{\smash{{\SetFigFont{8}{9.6}{\rmdefault}{\mddefault}{\updefault}3}}}}
\put(8033,-509){\makebox(0,0)[lb]{\smash{{\SetFigFont{8}{9.6}{\rmdefault}{\mddefault}{\updefault}2}}}}
\put(8033,293){\makebox(0,0)[lb]{\smash{{\SetFigFont{8}{9.6}{\rmdefault}{\mddefault}{\updefault}3}}}}
\put(7501,-102){\makebox(0,0)[lb]{\smash{{\SetFigFont{8}{9.6}{\rmdefault}{\mddefault}{\updefault}4}}}}
\put(8123,-88){\makebox(0,0)[lb]{\smash{{\SetFigFont{8}{9.6}{\rmdefault}{\mddefault}{\updefault}2}}}}
\end{picture}%
}
\caption{Conservative 1-flows and firings}\label{flow-ex}
\end{center}
\end{figure}

\end{example}

\section{Critical Groups as Chow Groups}
\label{sec:intersection}

An area for further research is to interpret the higher-dimensional
critical groups of a simplicial complex~$\Delta$ as simplicial
analogues of the Chow groups of an algebraic variety.  (For the
algebraic geometry background, see, e.g.,
\cite[Appendix~A]{Hartshorne} or \cite{Fulton}.)  We regard~$\Delta$
as the discrete analogue of a $d$-dimensional variety, so that
divisors correspond to formal sums of codimension-1 faces.  Even more
generally, algebraic cycles of dimension~$i$ correspond to simplicial
$i$-chains.  The critical group $K_i(\Delta)$ consists of closed
$i$-chains modulo conservative flows (in the language of
Section~\ref{sec:model}) is thus analogous to the Chow group of
algebraic cycles modulo rational equivalence.  This point of view has
proved fruitful in the case of graphs \cite{BHN,BN,HMY,Lor1,Lor2}.  In
order to develop this analogy fully, the next step
is to define a ring structure on
$\bigoplus_{i\geq 0}K_i(\Delta)$ with a ring structure analogous to that of
the Chow ring.  The goal is to define a ``critical ring'' whose
multiplication encodes a simplicial version of
intersection theory on~$\Delta$.

\bibliographystyle{alpha}

\end{document}